\documentclass[11pt]{article}
\usepackage[latin1]{inputenc}
\usepackage{amsmath,amssymb,amsthm}

\usepackage{bbm}
\usepackage{enumerate}

\usepackage[colorlinks=true,citecolor=black,linkcolor=black,urlcolor=blue]{hyperref}

\usepackage{overpic}
\usepackage{graphicx}
\usepackage{wrapfig}
\usepackage{rotating}

\usepackage{marginnote}

\usepackage[numbers]{natbib}
\usepackage[font=small,labelfont=bf]{caption}

\newtheorem{theorem}{Theorem}
\newtheorem{lemma}[theorem]{Lemma}

\newtheorem{corollary}[theorem]{Corollary}

\newtheorem{conjecture}[theorem]{Conjecture}

\newcommand{\Pb}{\mathbb{P}}

\newcommand{\N}{\mathbb{N}}
\newcommand{\R}{\mathbb{R}}
\newcommand{\Q}{\mathbb{Q}}
\newcommand{\Bin}{\mathrm{Bin}}
\newcommand{\rc}{\mathrm{rc}}
\newcommand{\diam}{\mathrm{diam}}
\newcommand{\Gnp}{\mathcal{G}(n,p)}

\newcommand{\mc}{\mathcal}

\title{On the threshold for rainbow connection number $r$ in random graphs}
\renewcommand{\leq}{\leqslant} 
\renewcommand{\geq}{\geqslant}
\renewcommand{\le}{\leqslant}
\renewcommand{\ge}{\geqslant}
\renewcommand{\epsilon}{\varepsilon}

\author{Annika Heckel and Oliver Riordan
\thanks{Mathematical Institute, University of Oxford,
24--29 St Giles', Oxford OX1 3LB, UK. E-mail: 
\texttt{$\{$heckel,riordan$\}$@maths.ox.ac.uk}
}
}

\begin{document}
\maketitle

\begin{abstract}
We call an edge colouring of a graph $G$ a \emph{rainbow colouring} if every pair of vertices is joined by a \emph{rainbow path}, i.e., a path where no two edges have the same colour. The minimum number of colours required for a rainbow colouring of the edges of $G$ is called the \emph{rainbow connection number} (or \emph{rainbow connectivity}) $\rc(G)$ of $G$. We investigate sharp thresholds in the Erd\H{o}s--R{\'e}nyi random graph for the property ``$\rc(G)\le r$'' where $r$ is a fixed integer. It is known that for $r=2$, rainbow connection number $2$ and diameter $2$ happen essentially at the same time in random graphs. For $r \ge 3$, we conjecture that this is not the case, propose an alternative threshold, and prove that this is an upper bound for the threshold for rainbow connection number $r$.
\end{abstract}

\section{Introduction}

The rainbow connection number is a measure of the connectivity of a graph introduced in 2008 by Chartrand, Johns, McKeon and Zhang \cite{chartrand:rainbow}, which has recently attracted the attention of a number of researchers (see \cite{li:rainbowsurvey}).

An edge colouring of a graph $G$ is called a \emph{rainbow colouring} if every pair of vertices is joined by a \emph{rainbow path}, i.e., a path where no two edges have the same colour. The minimum number of colours required for a rainbow colouring of the edges of $G$ is called the \emph{rainbow connection number} (or \emph{rainbow connectivity}) $\rc(G)$ of $G$.

In this paper we investigate the rainbow connection numbers of random graphs. As usual, we say that an event $E=E(n)$ holds \emph{with high probability} (whp) if $\lim_{n \rightarrow \infty} \Pb(E) =1$. For two functions $f, g: \N \rightarrow \R$, we write $f = O(g)$ if there are constants $C$ and $n_0$ such that $|f(n)| \le Cg(n)$ for all $n\ge n_0$, and $f = o(g)$ if $f(n)/g(n)\rightarrow 0$ as $n \rightarrow \infty$. Furthermore, we say $f = \Theta(g)$ if $f = O(g)$ and $g= O(f)$. Finally, we write $f= O^*(g)$ if there are constants $C$ and $n_0$ such that $|f(n)| \le (\log n)^C g(n)$ for all~$n\ge n_0$, where $\log n$ denotes the natural logarithm.

 Recall that the \emph{Erd\H{o}s--R{\'e}nyi random graph} with edge probability $p=p(n)$, denoted by $G \sim \Gnp$, is a graph with $n$ vertices where each of the ${n \choose 2}$ possible edges is included with probability $p$, independently of the others.  A sequence $p^*=p^*(n)$ is called a \emph{sharp threshold} for a graph property $\mc{P}$ if for all constants $c<1, C>1$, if $p(n) <cp^*(n)$ then whp $\Gnp \notin \mc{P}$, and if $p(n) >Cp^*(n)$ then whp $\Gnp \in \mc{P}$. Recall that for a (weak) threshold, the conditions assume $p=o(p^*)$ and $p^*=o(p)$, respectively. We also use the non-standard notion of a \emph{semisharp threshold}, which only requires the existence of some constants $c, C >0$ such that the properties above hold.

We are interested in the graph property
\[
 \mc{R}_r =\{ G: \rc(G) \le r\}
\]
of having rainbow connection number at most $r$. Caro, Lev, Roditty, Tuza and Yuster \cite{caro:rainbow} showed that $\sqrt{\frac{\log n}{n}}$ is a semisharp threshold for $\mc{R}_2$, and He and Liang \cite{he:rainbow} proved that for general $r \ge 2$, $\frac{(\log n)^{1/r}}{n^{1-1/r}}$ is a semi-sharp threshold for $\mc{R}_r$. As observed by Friedgut \cite{friedgut:sharpthresholdandksat}, a coarse threshold can only occur near rational powers of $n$. From Theorem 1.4 in \cite{friedgut:sharpthresholdandksat}, if the semi-sharp threshold for $\mc{R}_r$ were not sharp, there would be a sequence $(n_k)$ and $p(n_k) = \Theta \left( \frac{(\log n_k)^{1/r}}{n_k^{1-1/r}}\right)$ such that $b_1 n_k^\alpha \le p(n_k) \le b_2  n_k^\alpha$ for some constants $b_1, b_2 \in \R$ and $\alpha \in \Q$, a contradiction. Hence, a \emph{sharp threshold} for the property $\mc{R}_r$ exists.

Bollob\'as \cite{bollobas:diameter} showed that for any fixed $r \ge 2$, $\frac{(2\log n)^{1/r}}{n^{1-1/r}}$ is a sharp threshold for the graph property
\[
 \mc{D}_r = \{ G: \diam(G) \le r\}
\]
of having diameter at most $r$. In particular, $\mc{R}_r$ and $\mc{D}_r$ have the same weak threshold. It is an easy observation that for any connected graph $G$, $\rc(G) \ge \diam(G)$, and therefore $\frac{(2\log n)^{1/r}}{n^{1-1/r}}$ is a lower bound for the sharp threshold of $\mc{R}_r$.

In \cite{rainbowhittingtime}, we showed that for $r=2$, the hitting times of the properties $\mc{R}_2$ and $\mc{D}_2$ coincide whp, so in particular these properties have the same sharp threshold. It is a natural question to ask whether this result may be extended to $r \ge 3$. However, it seems that the situation for $r \ge 3$ is fundamentally different from the case $r=2$, and the methods used in \cite{rainbowhittingtime} do not carry over to the general case. In fact, there are good reasons to believe that the following may be the true sharp threshold for $\mc{R}_r$ where $r \ge 3$.

\begin{conjecture}
\label{thresholdconjecture}
Fix an integer $r \ge 3$, set $C= \frac{r^{r-2}}{(r-2)!}$, and let
\begin{equation}
\label{conjecturedthreshold}
p(n)=\frac{\left(C\log n\right)^{1/r}}{n^{1-1/r}}.
\end{equation}
Then $p(n)$ is a sharp threshold for the graph property $\mc{R}_r$.
\end{conjecture}

In one direction, there is a heuristic argument that (\ref{conjecturedthreshold}) is a lower bound for the sharp threshold for $\mc{R}_r$. Let $\epsilon>0$ and $p= \frac{\left(C(1-\epsilon)\log n\right)^{1/r}}{n^{1-1/r}}$, and colour the edges of $G\sim \Gnp$ with $r$ colours independently and uniformly at random. For a given pair of vertices $v,w$, the expected number of rainbow paths joining them is about $\frac{r!}{r^r}n^{r-1}p^r = (1-\epsilon)\left(1-\frac{1}{r}\right)\log n$, and the probability that $v$ and $w$ are not joined by any rainbow path is about $n^{-(1-\epsilon)\left(1-\frac{1}{r}\right)}$. The expected number of pairs of vertices not joined by any rainbow path is therefore $\Theta\left(n^{1+\frac{1}{r}+\epsilon\left(1-\frac{1}{r}\right)}\right)$. Assuming that these events behave roughly independently for different pairs of vertices, we would expect the overall probability of the random colouring being a rainbow colouring to be about $ \exp\left(-\Theta\left(n^{1+\frac{1}{r}+\epsilon\left(1-
\frac{1}{r}\right)}\right)\right) $. Conditional on $G$ having $O^*\left(n^{1+\frac{1}{r}}\right)$ edges, which holds with very high probability, there are $\exp \left(O^*\left(n^{1+\frac{1}{r}}\right)\right)$ possible edge colourings. The probability that there exists at least one rainbow colouring is then bounded by the total number of colourings multiplied with the probability that a random colouring is a rainbow colouring, which tends to~$0$.

The best known upper bound for the sharp threshold of the property $\mc{R}_r$ is $\frac{(2^{20r} \log n)^{1/r}}{n^{1-1/r}}$, which was shown in the proof of the semi-sharp threshold in \cite{he:rainbow}. 
We will establish the other direction of Conjecture \ref{thresholdconjecture} in this article, i.e., we will prove the following result.
\begin{theorem}
\label{upperboundtheorem}
 Fix an integer $r \ge 3$ and $\epsilon >0$. Set $p=p(n)=\frac{(C(1+\epsilon)\log n)^{1/r}}{n^{1-1/r}}$, and let $G \sim \Gnp$. Then whp, $\rc(G) = r$.
\end{theorem}

\section{Proof of Theorem \ref{upperboundtheorem}}

Let $G= (V,E)\sim \Gnp$. The basic idea of the proof is as follows. First we colour the edges of $G$ independently and uniformly at random using $r$ colours. We call a pair of vertices \emph{dangerous} if it is joined by at most $K$ rainbow paths of length $r$ in this colouring, where $K$ is a constant which will be defined later.

For each dangerous pair, we will select one path joining it and change the colours of the edges to make it a rainbow path, which will yield a rainbow colouring. To see that this is possible without any conflicts and that this does not rainbow-disconnect the pairs that previously had many rainbow paths, we need to study the structure of the graph and its dangerous pairs.

The proof is organised as follows. In Section \ref{preliminaries} some of the known bounds needed for our later work will be reviewed. Section \ref{generalobservations} contains general observations on the distribution of edges of each colour and paths of length $r$ in the randomly coloured graph. The heart of the proof is Section \ref{themainlemma}, where the key lemma is proved. This lemma ensures that when we later select a path of length $r$ for every dangerous pair of vertices and recolour it to make it a rainbow path, it is possible to do so without using any edges from a path that was previously assigned to another dangerous pair of vertices. Finally, in Section \ref{finishingtheproof}, the recolouring procedure will be described in detail and it will be shown that we can indeed find a rainbow $r$-colouring of the edges of $G$ with this strategy.

\subsection{Preliminaries}
\label{preliminaries}

We will need a number of bounds for the tail of binomial distributions, which are derived from the well-known Chernoff bounds (\cite{chernoffbound}, see also \cite[p.26]{janson:randomgraphs}):

\begin{lemma}
\label{chernoffbound}
Let $X$ be distributed binomially with parameters $n $ and $p $, and let $0<x<1$. 
\begin{enumerate}[(i)]
 \item If $x \geq p$, then $\Pb(X \geq nx) \leq \left[ \left( \frac{p}{x}  \right)^x \left( \frac{1-p}{1-x}\right)^{1-x} \right]^n$.
 \item If $x \leq p$, then $\Pb(X \leq nx) \leq \left[ \left( \frac{p}{x}  \right)^x \left( \frac{1-p}{1-x}\right)^{1-x} \right]^n$.
\end{enumerate}
\end{lemma}
\qed

The following corollary of the Chernoff bounds is given in Theorem 2.1 in \cite{janson:randomgraphs}.
\begin{corollary}
\label{newchernoffcorollary}
Let $X$ be distributed binomially with parameters $n$ and $p$, and let $\varphi(x) = (1+x) \log (1+x)-x$ for $x \ge -1$ and $\varphi(x)=-\infty$ otherwise. Then for all $t \ge 0$,
\[
 \Pb(X \leq np - t) \le  \exp\left(-np \varphi \left(-\frac{t}{np}\right)\right).
\]
\end{corollary}
\qed

In many applications it suffices to use the following more convenient bound (see Corollary 2.3 in \cite{janson:randomgraphs}).
\begin{corollary}
\label{chernoffcorollary}
Let $X$ be distributed binomially with parameters $n $ and $p $. If $0 <\epsilon \leq \frac{3}{2}$, then
\[
\Pb\left( \left| X - np \right| \geq \epsilon np \right) \leq 2 e^{-\epsilon^2 np /3}.
\]
\end{corollary}
\qed

The following consequence of the Chernoff bounds can be obtained by applying Lemma \ref{chernoffbound} to $X_i$ with $x_i = \frac{k}{n_i}$ (full details are given in \cite{rainbowhittingtime}).
\begin{corollary}
\label{binomialbound}
Let $(n_i)_{i \in \N}$ be a sequence of integers, and let $(p_i)_{i \in \N}$ be a sequence of probabilities. Let $X_i \sim \Bin (n_i, p_i)$, and let $k \in \N$ be constant. Suppose that $\mu_i := n_i p_i \rightarrow \infty$ as $i \rightarrow \infty$. Then
\[
\Pb(X_i \leq k) = O(\mu_i^k e^{-\mu_i}).
\]
\end{corollary}
\qed

We will also need the following easy observation on the probability of a union of events.

\begin{lemma}
\label{eventbounds}
 Let $A_i$, $i =1, \dots, k$, be events in a probability space $(\Omega, \mc{F}, \Pb)$. Then
\[
 \sum_{i} \Pb\left( A_i\right)  - \sum_{i} \sum_{ j < i} \Pb (A_i \cap A_j) \le \Pb\left( \bigcup_i A_i\right) \leq \sum_{i} \Pb\left( A_i\right) .
\]
\end{lemma}
\qed

\subsection{General observations}
\label{generalobservations}

For the rest of the paper, define $p=p(n)$ as in Theorem \ref{upperboundtheorem}, let $G \sim \Gnp$ and colour the edges of $G$ independently and uniformly at random using $r$ colours.

\begin{lemma}
\label{edgesofeachcolour}
Let $\delta >0$ be constant, let $W \subset V$ be a set of vertices with $|W| \sim n$, and let $v \in V$. Then for every colour, with probability $1-o\left(\exp\left(-n^{1/r}\right)\right)$, there are at least $\frac{1 - \delta}{r} np$ and at most $\frac{1 + \delta}{r} np$ edges between $v$ and $W$ of the given colour.
\end{lemma}
\begin{proof}
The number of such edges is distributed binomially with parameters $|W|$ (or $|W|-1$ if $v \in W$) and $p/r$. Since $|W|p/r \sim \left((1+\epsilon)C n \log n\right)^{1/r}/r$, the probability that there are fewer than $\frac{1 - \delta}{r} np$ or more than $\frac{1 + \delta}{r} np$ such edges is $o\left(\exp\left(-n^{1/r}\right)\right)$ by Corollary \ref{chernoffcorollary}.
\end{proof}

For $k \in \N$, we call a path of length $k$ in $G$ a \emph{$k$-path}, so a $k$-path is of the form $x_0 x_1 \dots x_k$ where the $x_i$ are distinct vertices. We call a collection of paths in the graph \emph{independent} if no two of them share any inner vertices.

\begin{lemma}
\label{enoughnormalpaths}
There is a constant $c>0$ such that whp every pair of vertices in $G$ is joined by at least $c \log n$ independent $r$-paths.
\end{lemma}
\begin{proof}
Fix two distinct vertices $v$ and $w$. We shall explore the $(r-1)$-neighbourhood of $v$ step-by-step, and apply Lemma \ref{edgesofeachcolour} at each step with a suitable $\delta >0$ to see that the sets we discover have the right size. Let $\delta>0$ be such that $(1-\delta)(1+\epsilon)^{1/r} = (1+\frac{\epsilon}{2})^{1/r}$. 

We start by considering all edges between $v$ and $W_1=V \setminus \{v,w\}$. By Lemma \ref{edgesofeachcolour}, with probability  $1-o\left(\exp\left(-n^{1/r}\right)\right)$ there are between $(1 - \delta) np$ and $(1 + \delta)np$ edges between $v$ and $W_1$.
 Condition on this, and denote by $N_1$ the set of vertices in $ W_1$ which are adjacent to $v$.

Next, let $W_2 = W_1 \setminus N_1$, and consider the edges between $N_1$ and $W_2$. Note that by our condition on the size of $N_1$, $|W_2| \sim n$. Furthermore, the edges between  $N_1$ and $W_2$ are disjoint from and therefore independent of the edges we have considered so far. We go through the vertices $z$ in $N_1$ one after the other, revealing the edges present. However, we disregard any edges to vertices which are adjacent to another vertex in $N_1$ which was considered earlier, so that the edges revealed form a tree. Applying Lemma~\ref{edgesofeachcolour} at each step, we can see that with probability $1-|N_1|o\left(\exp\left(-n^{1/r}\right)\right) =1 - o\left(\exp\left(-n^{1/2r}\right)\right)$, at each step there are between $(1 - \delta) np$ and $(1 + \delta) np$ edges.

Denote by $N_2$ the set of vertices in $W_2$ adjacent to a vertex in $N_1$, and let $W_3 = W_2 \setminus N_2$. We now proceed in the same way and explore the neighbours in $W_3$ of all vertices in $N_2$ disjointly, conditional on the neighbourhoods so far having the right sizes for all vertices
 according to Lemma \ref{edgesofeachcolour}.

We continue in this way until we have explored the entire $(r-1)$-neighbourhood $N_{r-1}$ of the vertex $v$ in $W_1$. Note that if all neighbourhoods have the right size, in total $O\left((np)^{r-1}\right)=o(n)$ vertices are revealed, so we can apply Lemma \ref{edgesofeachcolour} at each step. 
With probability $1-o(\exp \left(-n^{1/2r}\right)) $, we now have a tree with at least $\left( (1-\delta)np \right)^{r-1}=\left((1+\frac{\epsilon}{2})C n \log n\right)^{(r-1)/r}$ leaves. We can group the leaves together depending on which of the edges incident with $v$ their path to $v$ contains (i.e., which vertex in $N_1$ they originate from) --- each group has size at least  $\left((1+\frac{\epsilon}{2})C n \log n\right)^{(r-2)/r}$, and there are at least  $\left(\left(1+\frac{\epsilon}{2}\right)C n \log n\right)^{1/r}$ groups. The edges between $w$ and the leaves of the tree are independent from the edges that have been explored before. The probability that $w$ has a neighbour in a given vertex group is therefore at least 
\begin{align*}1-&(1-p)^{\left(\left(1+\frac{\epsilon}{2}\right)C n \log n\right)^{(r-2)/r}} \geq 1-\exp \left(-p\left(\left(1+\frac{\epsilon}{2}\right)C n \log n\right)^{(r-2)/r} \right) \\
&\geq 1-\exp \left(-\left(\left(1+\frac{\epsilon}{2}\right)C \log n\right)^{(r-1)/r} n^{-1/r}\right)\\
&\sim \left(\left(1+\frac{\epsilon}{2}\right)C  \log n\right)^{(r-1)/r}n^{-1/r},
\end{align*}
using the fact that $1-x \le \exp(-x)$ for all $x \in \R$ and that $1-\exp(-x) \sim x$ as $x \rightarrow 0$. These events are independent for the different vertex groups, so the number of groups with at least one edge to $w$ is distributed binomially. If we pick one edge from each such  vertex group, this gives independent paths from $v$ to $w$ by construction. Since there are at least  $\left((1+\frac{\epsilon}{2})C n \log n\right)^{1/r}$ vertex groups, the expected number of such paths is at least $(1+\frac{\epsilon}{3})C  \log n$ if $n$ is large enough. Note that $\varphi(x)\rightarrow 1$ as $x \searrow -1$, where $\varphi$ is the function defined in Corollary~\ref{newchernoffcorollary}. Therefore, since $r \geq 3$ and $C=\frac{r^{r-2}}{(r-2)!}>2$, if we pick $c>0$ small enough,  the probability that there are fewer 
than $c \log n$ independent $r$-paths joining $v$ and $w$ is $o(n^{-2})$ by Corollary~\ref{newchernoffcorollary}.\end{proof}

\subsection{The main lemma}
 \label{themainlemma}
Let 
\[L = \left\lceil \frac{17 r}{\epsilon (r-1)}\right\rceil, K=rL, \text{ and } S = 2Lr^2+2.
\]

 Call a pair of vertices \emph{dangerous} if it is joined by at most $K$ independent rainbow $r$-paths in the random colouring. The following lemma will form the main part of the proof.

\begin{lemma}
\label{maindistinct}
For a pair $\{v,w\}$ of vertices, denote by $A_{v,w}$ the event shown in Figure \ref{mainlemmapic}: There are $L$ independent $r$-paths $P_1,\ldots,P_L$ joining $v$ and $w$, and $L$ $r$-paths $Q_1,\ldots,Q_L$ such that, writing $\{x_i,y_i\}$ for
the end vertices of $Q_i$, the following conditions hold:
\begin{enumerate}[i)]
 \item For each $i$, $P_i$ contains an edge $e_i$ that is also on $Q_i$.
 \item The pairs $\{ x_i, y_i  \}$, $i=1, \ldots, L$, and $\{v,w\}$ are distinct.
 \item All pairs $\{ x_i, y_i  \}$, $i=1, \ldots, L$, are dangerous.
\end{enumerate}
Then whp $A_{v,w}$ does not hold for any pair $\{v,w\}$ of vertices.
\end{lemma}

The idea of the proof is the following. For one pair $\{ x_i, y_i \}$ as in the lemma, the expected number of rainbow $r$-paths joining $x_i$ and $y_i$ is roughly $\frac{r!}{r^r} n^{r-1} p^r = \frac{r!}{r^r} C (1+\epsilon) \log n = \frac{r-1}{r}(1+\epsilon) \log n$. Since the rainbow paths behave roughly binomially, the probability that $\{ x_i, y_i \}$ is dangerous is about $n^{-\frac{r-1}{r}(1+\epsilon)}(\log n)^K$ by Corollary \ref{binomialbound}.

Therefore, given $v$ and $w$, the probability that there is one path $P_i$ containing an edge $e_i$ which also lies on an $r$-path joining a dangerous pair $\{ x_i, y_i \}$ is about
\[
O^* \left( n^{2r-2} p^{2r-1} n^{-\frac{r-1}{r}(1+\epsilon)} \right)= O^*(n^{-\epsilon \frac{r-1}{r}} ) .
\]
Therefore, if we can show that these events do not depend on each other too much for different $P_i$, then we would expect that the overall probability that there are $L$ such paths is about $O^*\left(n^{-L\epsilon \frac{r-1}{r}}\right)$. If $L$ is chosen large enough, this will then be $o(n^{-2})$, completing the proof of the lemma.

A formal proof of this idea requires some care.

\begin{figure}[tb]
\begin{center}
\begin{overpic}[width=0.9\textwidth]{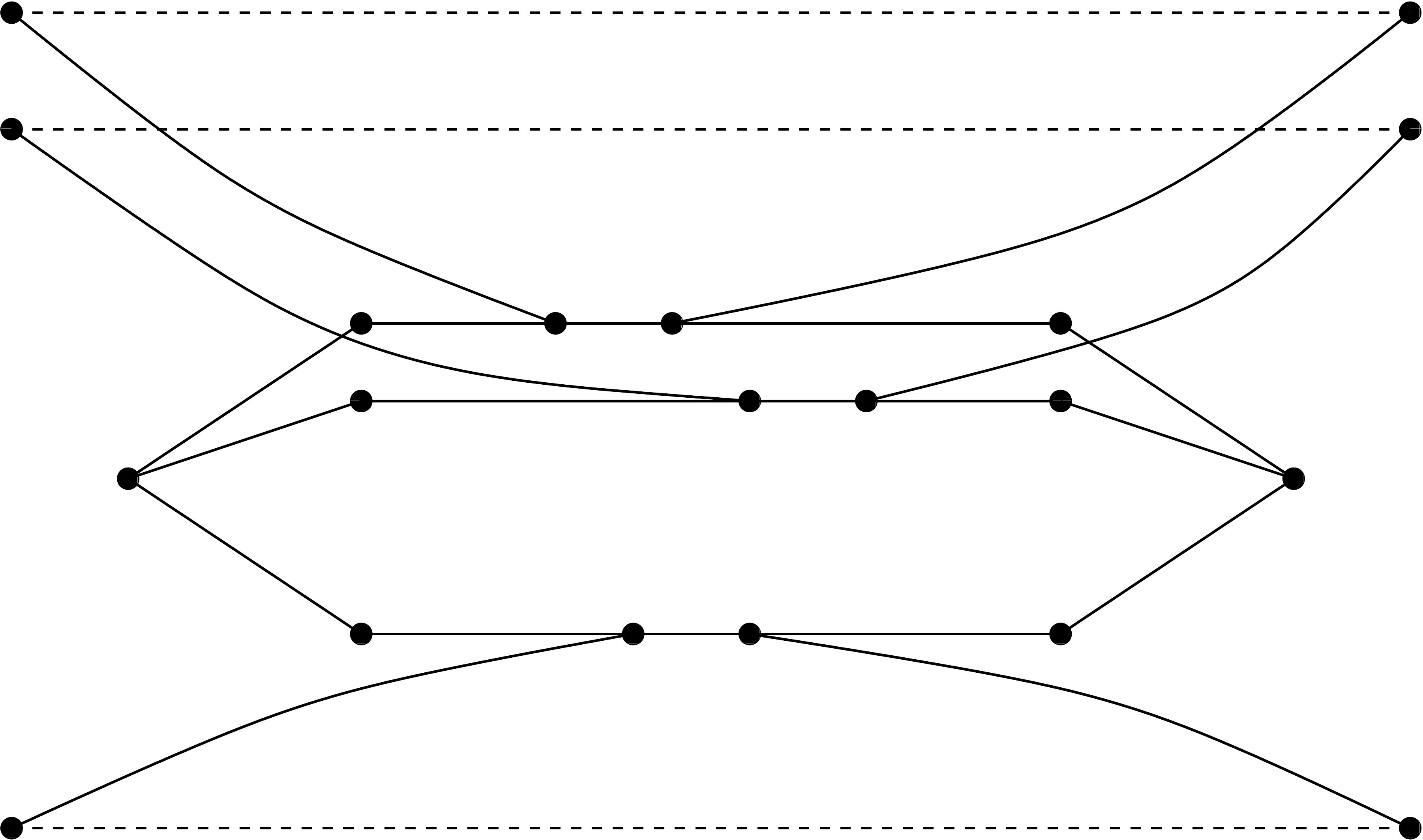}
\put(5.5,24.7){$v$}
\put(92.8,24.7){$w$}
\put(42,37.5){$e_1$}
\put(55.5,32){$e_2$}
\put(47,15.5){$e_L$}
\put(-3.7,59){$x_1$}
\put(-3.7,50.75){$x_2$}
\put(-3.7,1.5){$x_L$}
\put(100.5,59){$y_1$}
\put(100.5,50.75){$y_2$}
\put(100.5,1.5){$y_L$}
\put(14.8,32.2){$P_1$}
\put(16.5,25){$P_2$}
\put(14.8,17.5){$P_L$}
\put(18,46){$Q_1$}
\put(14.5,41){$Q_2$}
\put(14.8,4.6){$Q_L$}
\put(25.2,21){\begin{sideways}$\dots$\end{sideways}}
\put(74.3,21){\begin{sideways}$\dots$\end{sideways}}
\end{overpic}
\caption{The event $A_{v,w}$. Dashed lines show dangerous pairs. The paths $P_i$ only meet at $v$ and $w$, while the paths $Q_i$ may share vertices with each other and with the paths $P_i$. The pairs $\{x_i, y_i\}$ are distinct, but not necessarily disjoint. \label{mainlemmapic}}
\end{center}
\end{figure}
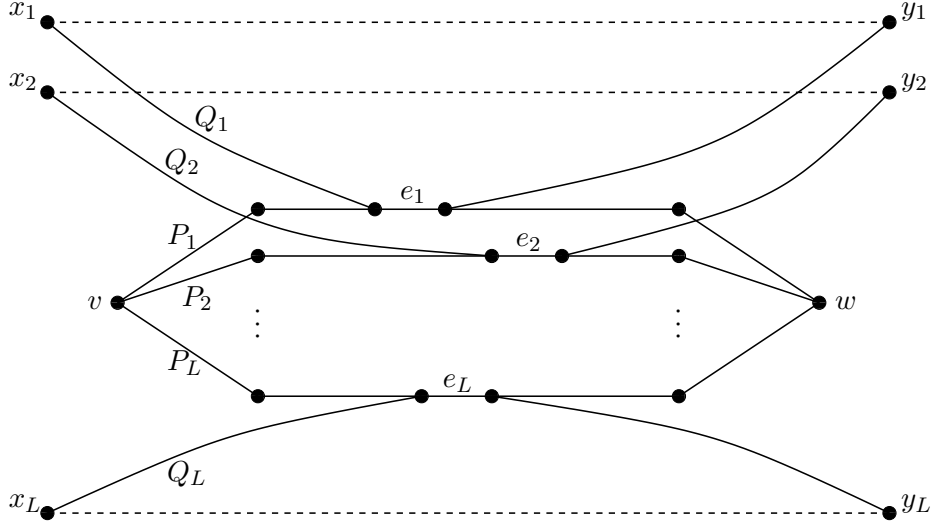

\begin{proof}[Proof of Lemma \ref{maindistinct}]
Fix distinct vertices $v$ and $w$. Consider a possible configuration of vertices and edges for the paths $P_i$, $Q_i$, edges $e_i$ and pairs $\{x_i, y_i\}$ as in conditions (i) and (ii) of $A_{v,w}$. Denote by $k$ the number of vertices in the configuration other than $v$ and $w$, and let $l$ be the number of edges in the configuration. Then $k \leq 2(r-1)L$, as the configuration consists of $L$ $r$-paths $P_i$ with endpoints $v$ and $w$, and $L$ $r$-paths $Q_i$ which each share at least two vertices with a path $P_i$.

Note that the configuration is connected, and if we remove one edge on all but one path $P_i$, it is still connected as there is still one $v$-$w$ path left. Since a connected graph with $m$ vertices has at least $m-1$ edges, it follows that $l-(L-1) \geq (k+2)-1$, so $l \geq k+L$. Therefore,
\begin{equation}
\label{configurationcontribution}
 n^k p^l \leq (np)^k p ^L \leq (np)^{2(r-1)L}p ^L = n^{2(r-1)L} p^{(2r-1)L}.
\end{equation}
Now condition on a specific such configuration being present in $G$. Let $W$ denote the set of vertices involved in the configuration, including $v$ and $w$,  and let $V' = V \setminus W$. Then $|W|=k+2\leq2Lr$, and  $|V'| \sim n$. The edges between $W$ and $V'$ and within $V'$ are disjoint from and therefore independent from the edges involved in the configuration.

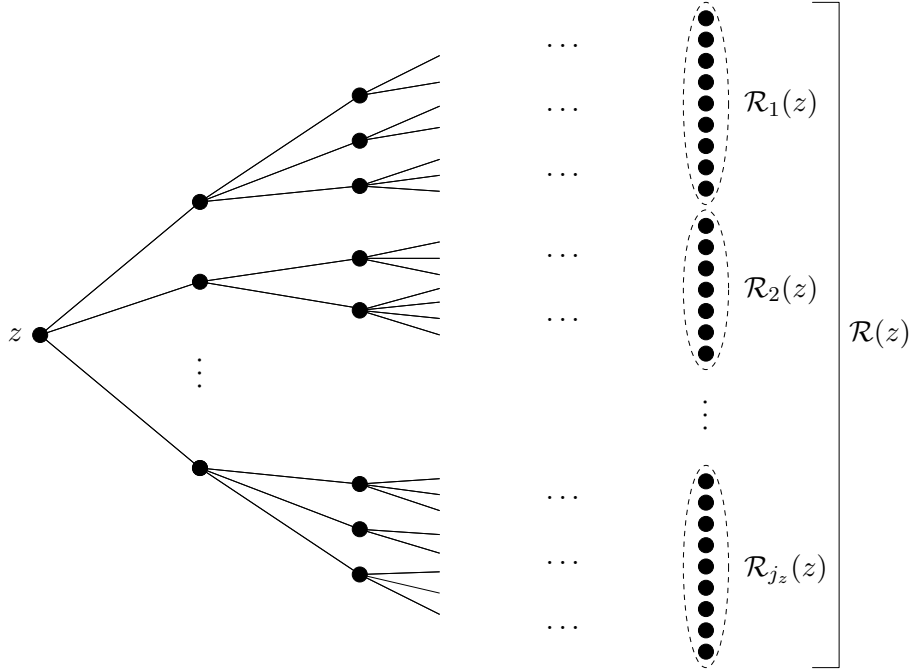
\begin{figure}[htb]
\begin{center}
\begin{overpic}[width=0.85\textwidth]{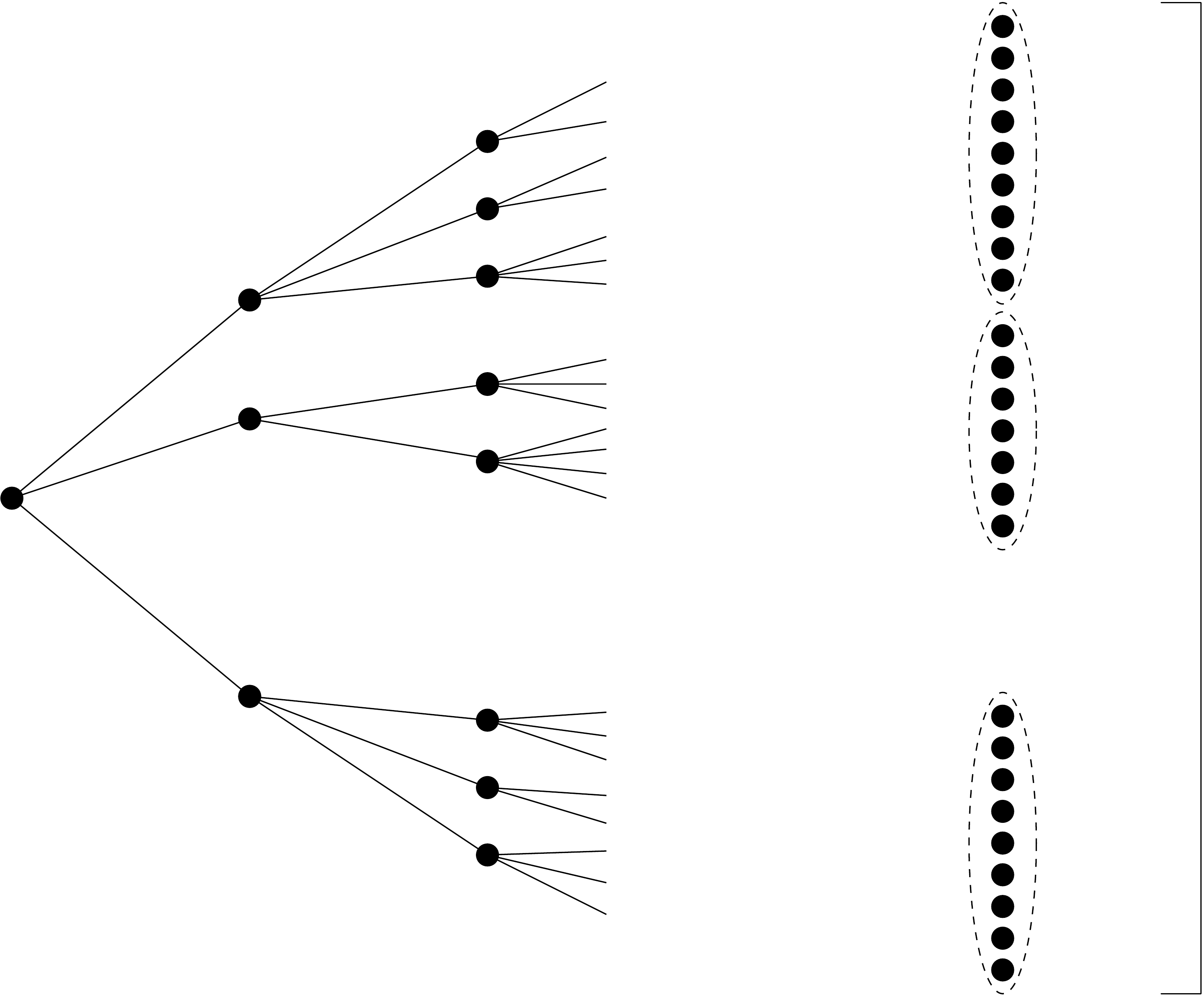}
\put(-3,40.6){$z$}
\put(20.5,34.5){\begin{sideways}$\dots$\end{sideways}}
\put(83,29.3){\begin{sideways}$\dots$\end{sideways}}
\put(63.5,77){$\dots$}
\put(63.5,69){$\dots$}
\put(63.5,61){$\dots$}
\put(63.5,51){$\dots$}
\put(63.5,43){$\dots$}
\put(63.5,21){$\dots$}
\put(63.5,13){$\dots$}
\put(63.5,5){$\dots$}
\put(101,40.6){$\mc{R}(z)$}
\put(88,69){$\mc{R}_1(z)$}
\put(88,46){$\mc{R}_2(z)$}
\put(88,11.5){$\mc{R}_{j_z}(z)$}
\end{overpic}
\caption{The tree of depth $s$ obtained by the exploration of the rainbow $s$-neighbourhood of $z$. The paths from the root $z$ to the leaves are rainbow paths. We have one such tree for each $z \in D$; these trees are disjoint.\label{rainbownbhd}}
\end{center}
\end{figure}

Let $s= \left\lfloor \frac{r-1}{2} \right\rfloor$, and let $D$ be the set of all vertices $x_i$ and $y_i$ from the configuration.  We now explore the $s$-neighbourhoods $\Gamma^s(z)$ of the vertices $z \in D$ within $V'$. We want to find disjoint subsets of $\Gamma^s(z)$ such that all elements are joined to $z$ by a rainbow $s$-path and all such paths are independent, except if they come from the same neighbour of $z$. We do this as in the proof of Lemma \ref{enoughnormalpaths} --- first explore the neighbours of $z_1 \in D$ in $V'$, then explore their neighbours in $V'$ and so on, then proceed with the next vertex $z_2 \in D$, and so on. As before, at each step, we disregard edges to vertices that have been explored already. Unlike in the proof of Lemma \ref{enoughnormalpaths}, at each step we only check for new neighbours joined by edges with colours not appearing on the path from $z$ to the current vertex. We group together vertices that come from the 
same edge incident with some $z \in D$.  This gives disjoint subsets $\mathcal{R}_j(z)$ of $V'$ for every $z \in D$, $1\leq j \leq j_z$, where the  
$\mathcal{R}_j(z)$ are the vertex groups which come from the same neighbour of $z$, as shown in Figure~\ref{rainbownbhd}. Let $\mc{R}(z) = \bigcup_{1 \leq j \leq j_z} \mc{R}_j(z)$.  Then by definition, the following properties hold:
\begin{enumerate}[i)]
\item For every $z_1$, $z_2$, $1\leq j_1 \leq j_{z_1}$ and $1\leq j_2 \leq j_{z_2}$ , if $z_1 \neq z_2$ or if $j_1 \neq j_2$ then the sets $\mathcal{R}_{j_1}(z_1) $ and $\mathcal{R}_{j_2}(z_2) $ are disjoint.
\label{enum1}
\item For every $z\in D$, every vertex $z' \in \mathcal{R}(z)$ is joined to $z$ by a rainbow $s$-path $P_{z'}$ with all inner vertices in $V'$.
\label{enum2}
\item For every $z_1$, $z_2$, $j_1$, $j_2$ and $z'_1 \in  \mathcal{R}_{j_1}(z_1)$, $z'_2 \in  \mathcal{R}_{j_2}(z_2)$, if $z_1 \neq z_2$ or if $j_1 \neq j_2$, then the paths $P_{z_1'}$ and $P_{z_2'}$ do not share any inner vertices.
\label{enum3}
\end{enumerate}
Applying Lemma \ref{edgesofeachcolour} at each step of our exploration with $\delta > 0$ such 
that $(1-\delta) (1+\epsilon)^{1/r} \ge \left(1+\frac{\epsilon}{2}\right)^{1/r}$ and $(1+\delta) (1+\epsilon)^{1/r}\le (1+2\epsilon)^{1/r}$, we see that with probability $1-o(n^{-2Lr})$, the following additional properties hold:
\begin{enumerate}[i)]
\setcounter{enumi}{3}
\item For all $z\in D$, $\left(  C \left(1+\frac{\epsilon}{2}\right) n \log n\right)^{1/r} \leq j_z \leq  \left(  C \left(1+2 \epsilon \right) n \log n\right)^{1/r}$.
\label{enum4}
\item For all $z \in D$ and $1 \le j \le j_z$, $\left|\mc{R}_j (z) \right| = O^* \left( n^{\frac{s-1}{r}}\right)$.
\label{enum4.5}
 \item For every subset $S$ of the $r$ available colours such that $|S|=s$ and for all $z\in D$ there are at least
\[
\frac{s!}{r^s} \left(  C \left(1+\frac{\epsilon}{2}\right) n \log n\right)^{s/r}
\]
and at most
\[
\frac{s!}{r^s} \left(  C \left(1+2 \epsilon \right) n \log n\right)^{s/r}
\]
vertices $z' \in \mathcal{R}(z)$ such that the colours appearing on $P_{z'}$ are exactly the colours in $S$.
\label{enum5}
\end{enumerate}

Assume (\ref{enum1}) -- (\ref{enum5}) from now on. Then $|\mc{R}|= O((n \log n)^{s/r}) = o(n)$, so $|V' \setminus \mc{R}| \sim n$.

\begin{enumerate}[]
 \item \textbf{Case 1: $r$ is odd.}
In this case $s = \left\lfloor \frac{r-1}{2}\right \rfloor = \frac{r-1}{2}$. For every subset $S$ of colours of size $s$, there are ${r-s \choose s} = s+1$ sets of colours  of size $s$ disjoint from $S$. Therefore, for every vertex $u_1$ in some set $\mathcal{R}(x_i)$ there are at least
\[
(s+1) \frac{s!}{r^s} \left(  C \left(1+\frac{\epsilon}{2}\right) n \log n\right)^{s/r}
\]
vertices $u_2$ in $\mathcal{R}(y_i)$ such that an edge $u_1 u_2$ of the correct colour would complete a rainbow $r$-path from $x_i$ to $y_i$. 

Therefore, there are at least 
\[
 {r \choose s} (s+1) \frac{(s!)^2}{r^{2s}} \left(  C \left(1+\frac{\epsilon}{2}\right) n \log n\right)^{2s/r} = \frac{r!}{r^{r-1}} \left(  C \left(1+\frac{\epsilon}{2}\right) n \log n\right)^{(r-1)/r}
\]
potential edges between $\mathcal{R}(x_i)$ and $\mathcal{R}(y_i)$ such that each one would complete a rainbow path from $x_i$ to $y_i$. 
Each of these edges is present and has the correct colour for a rainbow path with probability $\frac{1}{r}p$, independently from the edges that have been revealed so far.
Therefore, the number of such edges is distributed binomially with mean at least 
\[
\frac{r!}{r^r} C \left(1+\frac{\epsilon}{2}\right) \log n = \frac{r-1}{r} \left(1+\frac{\epsilon}{2}\right) \log n.
\]
If we denote by $E_i$ the event that there are at most $2S K$ edges of the correct colour between $\mathcal{R}(x_i)$ and $\mathcal{R}(y_i)$ to complete a rainbow path between $x_i$ and $y_i$, then by Corollary \ref{binomialbound},
\[
\Pb (E_i) = O^* \left( n^{- \left(1+\frac{\epsilon}{2}\right)\frac{r-1}{r}} 
\right).
\]
The events $E_i$ are independent for different pairs $\{x_i, y_i \}$ since the pairs $\{x_i, y_i \}$ are distinct and all sets $\mathcal{R}(x_i)$, $\mathcal{R}(y_i)$ are disjoint. Hence,
\begin{equation}
\label{atmost2SK}
\Pb \left( \bigcap_{1 \leq i \leq L} E_i\right) = O^* \left( n^{- L\left(1+\frac{\epsilon}{2}\right)\frac{r-1}{r}} 
\right)
 =O \left( n^{- L\left(1+\frac{\epsilon}{4}\right)\frac{r-1}{r}} \right) .
\end{equation}
For every $z \in D$ and $1\leq j \leq j_z$, let $B_j^z$ denote the event that there are at least $S$ edges between $\mc{R}_j(z)$ and $\mc{R}\setminus \mc{R}_j(z)$. Then
\[
\Pb(B_j^z) \le \left( |\mathcal{R}_j(z)| |\mathcal{R}| p\right)^S = O^* \left( n^{S (s-1)/r}    n ^{S s/r}  p^S  \right) = O^* \left(n^{-S/r} \right),
\]
as $|\mathcal{R}_j(z)| = O^*\left(n^{(s-1)/r}\right)$ and $|\mathcal{R}| = O^*\left(n ^{s/r}\right)$. Therefore, letting $B =  \bigcup_{(z, j): 1 \leq j \leq j_z} B_j^z$,
\begin{equation}
\label{anySvertex}
\Pb \left(B\right) = O^* \left( n^{1/r} n^{-S/r} \right)= o(n^{-2Lr}), 
\end{equation}
by choice of $S$.

If $B$ does not hold but $E_i$ does for some $1 \le i \le L$, then in particular the pair $\{ x_i, y_i\}$ is not dangerous. This is because we have at least $2S K$ edges of the correct colour between $\mathcal{R}(x_i)$ and $\mathcal{R}(y_i)$ to complete a rainbow path between $x_i$ and $y_i$, but there are at most $S$ such edges from each particular vertex group $\mc{R}_j(x_i)$ or $\mc{R}_j(y_i)$, so we can successively pick $K$ such edges between pairwise distinct vertex groups $\mc{R}_j(x_i)$ or $\mc{R}_j(y_i)$, yielding $K$ independent rainbow paths between $x_i$ and $y_i$ by property (\ref{enum3}).

Therefore, if all $L$ pairs $\{x_i, y_i \}$ are dangerous, then $B \cup \bigcap_{1 \leq i \leq L} E_i$ holds.  Hence, by (\ref{atmost2SK}) and (\ref{anySvertex}), the probability that all $L$ pairs $\{x_i, y_i \}$ are dangerous is bounded by
\[
 O \left( n^{- L(r-1)\left(1+\frac{\epsilon}{4}\right)/r} \right) + o(n^{-2Lr}).
\]

 \item \textbf{Case 2: $r$ is even.}

In this case $s = \left\lfloor \frac{r-1}{2}\right \rfloor = \frac{r}{2}-1$. Let $u\in V' \setminus \mc{R}$. Given a vertex $u_1$  in some $\mathcal{R}(x_i)$, there are at least  
\[{{r-s} \choose s}\frac{s!}{r^s} \left(  C \left(1+\frac{\epsilon}{2}\right) n \log n\right)^{s/r}
\]
and at most
\[{{r-s} \choose s}\frac{s!}{r^s} \left(  C \left(1+2\epsilon \right) n \log n\right)^{s/r}
\]
vertices $u_2$ in $\mathcal{R}(y_i)$ such that adding edges $u_1 u$ and $u u_2$ of appropriate colours would complete a rainbow $r$-path from $x_i$ to $y_i$ via $u_1$, $u$ and $u_2$. Therefore, there are at least 
\begin{equation}
 {r \choose s} {{r-s} \choose s} \frac{s!^2}{r^{2s}} \left(  C \left(1+\frac{\epsilon}{2}\right) n \log n\right)^{2s/r} = \frac{r!}{2r^{r-2}} \left(  C \left(1+\frac{\epsilon}{2}\right) n \log n\right)^{\frac{r-2}{r}}
\label{numberofevents}
\end{equation}
and at most
\begin{equation}
\label{numberofevents2}
\frac{r!}{2r^{r-2}} \left(  C \left(1+2\epsilon\right) n \log n\right)^{\frac{r-2}{r}}
\end{equation}
pairs of vertices $u_1 \in \mathcal{R}(x_i)$, $u_2 \in \mathcal{R}(y_i)$ such that edges $u_1 u$ and $u u_2$ of appropriate colours would complete a rainbow path from $x_i$ to $y_i$. For one such pair $\{u_1, u_2\}$ and $u \in V' \setminus \mc{R}$, denote by $M_u^{u_1,u_2 }$ the event that the edges $u_1u$ and $uu_2$ are present and have one of the two possible colour combinations. Then 
\begin{equation}
\Pb(M_u^{u_1,u_2 }) = \frac{2}{r^2}p^2.
\label{oneevent}
\end{equation}
Moreover, if the events $M_u^{u_1,u_2 }$ and $M_u^{u'_1,u'_2 }$ hold for different pairs $\{u_1, u_2\}$ and $\{u'_1, u'_2\}$, then $u$ is adjacent to three or more distinct vertices from $\{u_1, u_2, u'_1, u'_2\}$, so \begin{equation}
\Pb \left( M_u^{u_1,u_2 } \cap M_u^{u'_1,u'_2 } \right) = O(p^3).                                                                                                                                                                                                                                                                                                                                                                                                                                                                                                                                                                                                                                               
\label{twoevents}                                                                                                                                                                                                                                                                                                                                                                                                                                                                                                                                                                                                                                                \end{equation}
For a vertex  $u \in V' \setminus \mc{R}$, denote by $F_u$ the event that $u$ is the middle vertex of any path as above for any $1 \leq i \le L$. Then by Lemma~\ref{eventbounds} and (\ref{numberofevents}), (\ref{numberofevents2}), (\ref{oneevent}), (\ref{twoevents}), 
\begin{align*}
\Pb(F_u) &\ge  L \frac{r!}{r^r}   \left( C \left(1+\frac{\epsilon}{2}\right) n \log n\right)^{\frac{r-2}{r}} p^2 - O^*\left(n ^{\frac{2(r-2)}{r}}\right) O\left(p^3\right) \\
&\ge L \frac{r-1}{r} \left(1+\frac{\epsilon}{2}\right) n^{-1} \log n  - O^*\left(n^{-1-\frac{1}{r}} \right)\\
&\sim L \frac{r-1}{r} \left(1+\frac{\epsilon}{2}\right) n^{-1} \log n .
\end{align*}
The events $F_u$ are independent for different $u \in V' \setminus \mathcal{R}$. Thus, since $|V' \setminus \mc{R}| \sim n$, the number of events $F_u$ that hold is distributed binomially with mean asymptotically at least 
\[ L \frac{r-1}{r} \left(1+\frac{\epsilon}{2}\right)\log n .
\]
By Corollary \ref{binomialbound}, the probability of the event $F$ that at most $2KLS$ of the events $F_u$ hold is
\begin{equation}
\label{atmost2KLT}
\Pb(F) = O^* \left( n^{-{(1+o(1)) L \frac{r-1}{r} \left(1+\frac{\epsilon}{2}\right)}} \right) 
= O\left(n^{-  L\frac{r-1}{r} \left(1+\frac{\epsilon}{4}\right)}\right)
.
\end{equation}
For every $z \in D$ and $1\leq j \leq j_z$, denote by $\tilde B^z_j$ the event that there are at least $S$ independent $2$-paths from (not necessarily distinct) vertices in $\mathcal{R}_j(z)$ to (not necessarily distinct) vertices in $\mathcal{R}$ with middle vertices in $V' \setminus \mc{R}$. Then, since $\left|\mathcal{R}_j(z)\right| = O^*\left( n ^{(s-1)/r} \right)$ and $\left|\mathcal{R}\right| = O^*\left( n ^{s/r} \right)$,
\begin{equation*}
\Pb(\tilde B_j^z) \le \left( \left|\mathcal{R}_j(z)\right|\left| V\right|\left|\mathcal{R}\right| p^2\right)^S= O^*\left(\left( n ^{\frac{s-1}{r}+1 + \frac{s}{r}} p^2 \right)^S \right) =  O^*\left(n^{-S/r} \right).
\end{equation*}
Therefore, if we let $\tilde B = \bigcup_{(z,j): 1\leq j \leq j_z} \tilde B^z_j$, then
\begin{equation}
 \label{anyTvertex}
\Pb(\tilde B) = O^* \left( n ^{1/r}  n^{-S/r}  \right) = o(n^{-2Lr}),
\end{equation}
by choice of $S$. 

If neither $\tilde B$ nor $F$ holds, then one of the pairs $\{x_i, y_i\}$ is not dangerous. This is because there are more than $2KLS$ vertices $u\in V' \setminus \mc{R}$ which are the middle vertices of rainbow paths joining pairs $\{x_i, y_i\}$, so there is an index $1\leq i_0 \leq L$ such that there are more than $2KS$ vertices $u \in V' \setminus \mc{R}$ which are the middle vertices of rainbow $r$-paths joining the pair $\{x_{i_0}, y_{i_0}\}$ (we can just pick the index with the maximum number of such vertices $u$). If $\tilde B$ does not hold, at most $S$ of those paths pass through any particular vertex group $\mathcal{R}_j(x_{i_0})$ or $\mathcal{R}_j(y_{i_0})$. Hence, we can successively select more than $K$ rainbow $r$-paths joining $\{x_{i_0}, y_{i_0}\}$ which pass through pairwise distinct vertex groups. These rainbow paths are independent by property (\ref{enum3}), so $\{x_{i_0}, y_{i_0}\}$ is not dangerous.

Hence, if all pairs $\{x_i, y_i\}$ are dangerous, then $\tilde B$ or $F$ holds, which by (\ref{atmost2KLT}) and (\ref{anyTvertex}) has probability

\begin{equation*}
O\left(n^{- L\left(1+\frac{\epsilon}{4}\right)\frac{r-1}{r} }\right)+o\left(n^{-2Lr}\right).
\end{equation*}

\end{enumerate}
So in each case, conditional on a configuration of paths $P_i$ and $Q_i$, edges $e_i$ and pairs $\{x_i, y_i\}$ as in conditions (\ref{enum1}) and (\ref{enum2}) of the event $A_{ v,w}$, the probability that all pairs $\{x_i, y_i \}$ are dangerous is at most 
\[
O\left(n^{- L\left(1+\frac{\epsilon}{4}\right)\frac{r-1}{r} }\right)+o\left(n^{-2Lr}\right).
\]
Together with (\ref{configurationcontribution}), it follows that the overall probability of $A_{v,w}$ is at most
\begin{align*}
\lefteqn{O \left(n^{2(r-1)L} p^{(2r-1)L} n^{- L\left(1+\frac{\epsilon}{4}\right)\frac{r-1}{r}} \right) +o(n^{-2L}) = } \\
&= O \left( \left(n^{2r-2 - \left(1+\frac{\epsilon}{4}\right) \frac{r-1}{r}} p^{2r-1}  \right)^L\right) +o(n^{-2})= O \left(n^{- \frac{\epsilon(r-1)}{8r}L }  \right)+o(n^{-2})\\&= o(n^{-2}),
\end{align*}
by choice of $L$.
So whp, there is no such pair $\{v,w\}$.
\end{proof}

\subsection{Completing the proof}
\label{finishingtheproof}

To finish the proof, we want to construct a rainbow colouring of the edges of $G$ from the given random colouring. By Lemmas \ref{enoughnormalpaths} and \ref{maindistinct}, we can assume that every pair of vertices is joined by at least $c \log n$ independent $r$-paths for a constant $c >0$, and that $A_{v,w}$ does not hold for any pair $\{v,w\}$ of vertices.

Recall that we call a pair of vertices dangerous if it is joined by at most $K$ independent rainbow $r$-paths  in the original random colouring. Take an arbitrary ordering of the dangerous pairs. We will go through them one by one, each time selecting an $r$-path joining the dangerous pair, changing its colours if necessary to make it a rainbow path, then \emph{flagging} all edges on the path to ensure they do not get recoloured later on.

Let $\{v,w\}$ be the pair we consider. It is joined by at least $c \log n \geq r L$ independent $r$-paths if $n$ is large enough. We want to find one such path where no edge is flagged yet.

So take a set $\mc{I}$ of $r L$ independent $r$-paths joining $v$ and $w$, and consider any such path $P_1$ in $\mc{I}$. Either none of its edges is flagged --- in this case, we have found our path. Otherwise, it contains (at least) one edge which is also on an $r$-path joining a dangerous pair other than $\{v,w\}$. For this dangerous pair, one path of length $r$ was flagged previously. Therefore, at most $r-1$ of the other paths in $\mc{I}$ can contain edges flagged for the same dangerous pair. Discard those paths and $P_1$. We are left with at least $r (L-1)$ paths joining $v$ and $w$. Select any such path $P_2$ and proceed in the same way as with $P_1$: Either $P_2$ is completely unflagged, or we remove $P_2$ and any other path with edges flagged for the same dangerous pair as $P_2$ from consideration, and are left with at least $r(L-2)$ paths. We repeat this procedure until we find a completely unflagged path. This happens at $P_L$ at the latest. Otherwise, if $P_L$ also contains an edge flagged for a new 
dangerous pair, then 
$A_{v,w}$ holds, a contradiction. 

Therefore, there is a path joining $\{v,w\}$ where no edge is flagged at the time we consider $\{v,w\}$. Select this path, change its colours if necessary to make it a rainbow path, then flag all its edges and move on to the next dangerous pair. Repeat this procedure until all dangerous pairs have been assigned rainbow paths.
                                                                                                                                                                                                                                                            
It only remains to check that during our recolouring procedure no previously non-dangerous pair has lost all of its rainbow paths. Let $\{v,w\}$ be a pair that was not dangerous before we started recolouring. Since it was originally joined by at least $K=rL$ rainbow paths, by the same argument as above for dangerous pairs, one of these paths must be completely unflagged, otherwise $A_{v,w}$ would hold. This path has retained its original colours and is therefore still a rainbow path. So all pairs of vertices are joined by rainbow paths now.
\qed

\bibliographystyle{plainnat}

\end{document}